\documentclass{article}

\usepackage{amsmath}
\usepackage{amssymb}
\usepackage{amsthm}

\newtheorem{thm}{Theorem}
\newtheorem{conj}{Conjecture}
\newtheorem{lem}{Lemma}
\newtheorem{prop}{Proposition}

\title{Variance of squarefull numbers in short intervals II}
\author{Tsz Ho Chan}
\date{}

\begin{document}
\maketitle

\begin{abstract}
In this paper, we continue the study on variance of the number of squarefull numbers in short intervals $(x, x + 2 \sqrt{x} H + H^2]$ with $X \le x \le 2X$. We obtain the expected asymptotic for this variance over the range $X^\epsilon \le H \le X^{0.180688...}$ unconditionally and over the optimal range $X^\epsilon \le H \le X^{0.25 - \epsilon}$ conditionally on the Riemann Hypothesis or the Lindel\"{o}f Hypothesis.
\end{abstract}

\section{Introduction and Main Results}

A number $n$ is {\it squarefree} if $p^2 \nmid \; n$ for every $p \mid n$, i.e. all the exponents in its prime factorization are exactly one. In contrast, a number $n$ is {\it squarefull} or {\it powerful} if $p^2 \mid n$ for every prime $p \mid n$, i.e. all the exponents in its prime factorization are at least two. It is well-known that every squarefull number $n$ can be factored uniquely as $n = a^2 b^3$ for some integer $a$ and squarefree number $b$. Let $Q(x)$ denote the number of squarefull numbers which are less than or equal to $x$. 
Then
\begin{equation} \label{squarefull}
Q(x) = \sum_{a^2 b^3 \le x} \mu^2(b)
\end{equation}
where $\mu(n)$ is the M\"{o}bius function. Bateman and Grosswald \cite{BG} proved that
\begin{equation} \label{fullcount}
Q(x) = \frac{\zeta(3/2)}{\zeta(3)} x^{1/2} + \frac{\zeta(2/3)}{\zeta(2)} x^{1/3} + O \bigl( x^{1/6} e^{-C_1 (\log x)^{4/7} (\log \log x)^{-3/7}} \bigr)
\end{equation}
for some absolute constant $C_1 > 0$ where $\zeta(s)$ is the Riemann zeta function. We are interested in squarefull numbers in short intervals, namely
\[
Q(x + y) - Q(x) \; \; \text{ with } \; \; y = o(x).
\]
Bateman and Grosswald's result \eqref{fullcount} yields
\begin{equation} \label{shortcount1}
Q \bigl(x + x^{1/2 + \theta} \bigr) - Q(x) \sim \frac{\zeta(3/2)}{2 \zeta(3)} x^\theta
\end{equation}
for $1/6 = 0.1666... \le \theta < 1/2$. The current best range of validity for \eqref{shortcount1} was obtained by Trifonov \cite{T}, showing that the asymptotic formula \eqref{shortcount1} is true for $19 / 154 = 0.1233... < \theta < 1/2$.

\bigskip

Inspired by a recent breakthrough work of Gorodetsky, Matom\"{a}ki, Radziwi\l{}\l{} and Rodgers \cite{GMRR} on variance of squarefree numbers in short intervals and arithmetic progressions, the author \cite{C} studied the variance of squarefull numbers in short intervals:
\[
\frac{1}{X} \int_{X}^{2 X} \Big| Q(x + y) - Q(x) - \frac{\zeta(3/2)}{\zeta(3)} (\sqrt{x + y} - \sqrt{x}) \Big|^2 dx
\]
and proved
\begin{thm} \label{thm1}
Given $\epsilon > 0$, $X > 1$ and $X^{4\epsilon} \le H \le X^{1/2 - \epsilon}$. With $y := 2 \sqrt{x} H + H^2$,
\[
\frac{1}{X} \int_{X}^{2 X} \Big|Q(x + y) - Q(x) - \frac{\zeta(3/2)}{\zeta(3)} H \Big|^2 dx \ll_\epsilon H^{16 / 11 + \epsilon}.
\]
Note that $16 / 11 = 1.454545\ldots$ and $\sqrt{x + y} - \sqrt{x} = H$ (which explains our choice of $y$).
\end{thm}
This implies that, for any $0 < \theta < 1/2$, the asymptotic formula \eqref{shortcount1} is true for almost all $x \in [X, 2X]$. The author also made the following
\begin{conj} \label{conj1}
Given $\epsilon > 0$, $X > 1$ and $X^\epsilon \le H \le X^{1/4 - \epsilon}$. With $y := 2 \sqrt{x} H + H^2$,
\[
\frac{1}{X} \int_{X}^{2 X} \Big|Q(x + y) - Q(x) - \frac{\zeta(3/2)}{\zeta(3)} H \Big|^2 dx \sim \frac{4 \zeta(4/3)}{3 \zeta(2)} \int_{0}^{\infty} \Bigl(\frac{\sin \pi y}{\pi y}\Bigr)^2 y^{1/3} dy \cdot H^{2/3}.
\]
\end{conj}
The reason for the exponent $1/4 - \epsilon$ is that the secondary main term from \eqref{fullcount} gives
\[
\frac{\zeta(2/3)}{\zeta(2)} (\sqrt[3]{x+y} - \sqrt[3]{x}) \gg \frac{y}{X^{2/3}} \gg \frac{H}{X^{1/6}} \gg H^{1/3}
\]
when $H \ge X^{1/4}$. Assuming the Riemann Hypothesis (RH) or the Lindel\"{o}f Hypothesis $\zeta(1/2 + it) \ll_\epsilon t^\epsilon$ (LH), the author \cite{C} proved that Conjecture \ref{conj1} is true for $X^\epsilon \le H \le X^{3 / 14 - \epsilon}$. Note that $3/14 = 0.2142857 \ldots$.

\bigskip

In this paper, we obtain unconditionally
\begin{thm} \label{thm2}
Given $0 < \epsilon < 0.01$. For $X^{\epsilon} \le H \le X^{\frac{189}{1046} - 2 \epsilon}$, we have
\[
\frac{1}{X} \int_{X}^{2X} \Big| Q(x+y) - Q(x) - \frac{\zeta(3/2)}{\zeta(3)} H \Big|^2 dx
= \frac{4 \zeta(4/3)}{3 \zeta(2)} \int_{0}^{\infty} S(y)^2 y^{1/3} dy \cdot H^{2/3} + O_\epsilon ( H^{2/3 - \epsilon / 16})
\]
where $S(x) = \frac{\sin \pi x}{\pi x}$. Note that $\frac{189}{1046} = 0.180688...$.
\end{thm}
Conditionally, we have
\begin{thm} \label{thm3}
Conjecture \ref{conj1} is true under the RH or the LH.
\end{thm}

To prove Theorem \ref{thm2}, we split the squarefull numbers into three ranges according to the size of $b$. For small $b$, we apply smooth weights, count by studying gaps between fractions, and obtain
\begin{prop} \label{prop1}
Given $0 < \epsilon < 0.01$. For $X^\epsilon \le H \le X^{3/10 - 3 \epsilon}$, we have
\begin{align*}
\frac{1}{X} \int_{X}^{2X} & \bigg| \mathop{\sum_{x < a^2 b^3 \le x + y}}_{b \le H^{2/3 + \epsilon}} \mu^2(b) - H \sum_{b \le H^{2/3 + \epsilon}} \frac{\mu^2(b)}{b^{3/2}} \biggr|^2 dx \\
=& (1 + O(X^{-\epsilon /2})) 2 H^2  \sum_{b \le H^{2/3 + \epsilon}} \frac{\mu^2(b)}{b^3} \sum_{n \ge 1} S\Bigl( \frac{n H}{b^{3/2}} \Bigr)^2 + O_\epsilon ( H^{2/3 - \epsilon / 2} )
\end{align*}
where $S(x) = \frac{\sin \pi x}{\pi x}$.
\end{prop}
\begin{prop} \label{prop3}
Given $0 < \epsilon < 0.01$. For $1 \le H \le X^{1/2}$, we have
\[
2 H^2 \sum_{b \le H^{2/3 + \epsilon}} \frac{\mu^2(b)}{b^3} \sum_{n \ge 1} S\Bigl( \frac{n H}{b^{3/2}} \Bigr)^2= H^{2/3} \cdot \frac{4 \zeta(4/3)}{3 \zeta(2)} \int_{0}^{\infty} S(y)^2 y^{1/3} dy + O_\epsilon ( H^{2/3 - \epsilon/6} )
\]
where $S(x) = \frac{\sin \pi x}{\pi x}$.
\end{prop}

For medium size $b$, similar to \cite{GMRR}, we transform the counting problem into contour integrals over Dirichlet polynomials and $\zeta(s)$, and use mean-value and large-value theorems on Dirichlet polynomials as well as bounds on $\zeta(s)$ and exponential sum techniques to obtain
\begin{prop} \label{prop2}
Given $0 < \epsilon < 0.01$ and $\lambda = 2/9 - \epsilon / 3$. For $X^\epsilon \le H \le X^{\frac{189}{1046} - 2 \epsilon}$, we have
\[
\frac{1}{X} \int_{X}^{2X} \bigg| \mathop{\sum_{x < a^2 b^3 \le x + y}}_{H^{2/3 + \epsilon} < b \le X^{1/3} / H^{\lambda}} \mu^2(b) - H \sum_{H^{2/3 + \epsilon} < b \le X^{1/3} / H^{\lambda}} \frac{\mu^2(b)}{b^{3/2}} \biggr|^2 dx \ll H^{2/3 - \epsilon / 8}.
\]
\end{prop}

For large $b$, we simply estimate the situation trivially with a certain parameter $\lambda$. It is worth mentioning that one can use Erd\H{o}s-Tur\'{a}n-Koksma inequality on equidistribution and exponential sum techniques to obtain a slightly bigger value for $\lambda$. This can extend the range of $H$ for Theorem \ref{thm2} slightly. Also, the interested readers may apply our proof of Theorem \ref{thm2} to get a better upper bound (such as $H^{4/3 + \epsilon}$) for Theorem \ref{thm1}. With the resolution of Conjecture \ref{conj1} under RH or LH, the next step of this study would be to investigate the following more precise variance:
\[
\frac{1}{X} \int_{X}^{2 X} \Big| Q(x + y) - Q(x) - \frac{\zeta(3/2)}{\zeta(3)} (\sqrt{x + y} - \sqrt{x}) - \frac{\zeta(2/3)}{\zeta(2)} (\sqrt[3]{x+y} - \sqrt[3]{x}) \Big|^2 dx.
\]

The paper is organized as follows. First, we collect some needed lemmas. Then, we perform an initial manipulation on the variance and estimate the large $b$ contribution. Next, we deal with small $b$ and prove Proposition \ref{prop1}. After that, we study medium size $b$ and obtain Proposition \ref{prop2}. Then, we finish the proof of Theorem \ref{thm2} and indicate modifications of the argument to yield Theorem \ref{thm3}. At the end, we give a proof of Proposition \ref{prop3}.

\bigskip

{\bf Notation.} We use $[x]$, $\{x\}$ and $\| x \|$ to denote the integer part of $x$, the fractional part, and the distance between $x$ and the nearest integer respectively. For any set $\mathcal{S}$, $|\mathcal{S}|$ stands for its cardinality. The symbols $f(x) = O(g(x))$, $f(x) \ll g(x)$ and $g(x) \gg f(x)$ are equivalent to $|f(x)| \leq C g(x)$ for some constant $C > 0$. The symbol $f(x) \asymp g(x)$ means that $C_1 f(x) \le g(x) \le C_2 f(x)$ for some constants $0 < C_1 < C_2$. The symbol $f(x) \sim g(x)$ means that $\lim_{x \rightarrow \infty} f(x)/g(x) = 1$. Finally, $f(x) = O_{\kappa} (g(x))$, $f(x) \ll_{\kappa} g(x)$, $g(x) \gg_{\kappa} f(x)$ or $g(x) \asymp_{\kappa} f(x)$ mean that the implicit constant may depend on $\kappa$.
\section{Some lemmas}

We need a lemma on cancellation of oscillatory integral over certain smooth function.
\begin{lem} \label{lem1}
Suppose $K > 1$ is some large integer. Let $H \ge 1$, $W, Q, L, R > 0$ satisfying $Q R / \sqrt{H} > 1$ and $R L > 1$. Suppose that $w(t)$ is a smooth function with support on $[\alpha, \beta]$ satisfying
\[
w^{(j)}(t) \ll_j W L^{-j} \; \; \text{ for } j = 0, 1, 2, ..., 2K-1.
\]
Suppose $h(t)$ is a smooth function on $[\alpha, \beta]$ such that $|h'(t)| \ge R$ and
\[
h^{(j)}(t) \ll_j H Q^{-j}, \; \; \text{ for } j = 2, 3, ..., 2K- 1.
\]
Then the integral
\[
I = \int_{-\infty}^{\infty} w(t) e^{i h(t)} dt
\]
satisfies
\[
I \ll_K (\beta - \alpha) W [ (Q R / \sqrt{H})^{-K} + (R L)^{-K} ].
\]
\end{lem}
\begin{proof}
This is basically Lemma 8.1 in \cite{BKY} or Lemma 1 in \cite{C}.
\end{proof}

We also need results on mean value and large value of Dirichlet polynomials, fourth-moment and subconvexity bound of the Riemann zeta function, and Van der Corput's process B (truncated Poisson summation).

\begin{lem}[Mean-value theorem] \label{lem-meanvalue}
Let $D(s) = \sum_{n = 1}^{N} a_n n^{-s}$ be a Dirichlet polynomial. Then
\[
\int_{0}^{T} |D(i t)|^2 dt = (T + O(N)) \sum_{n = 1}^{N} |a_n|^2.
\]
\end{lem}
\begin{proof}
See \cite[Chapter 7]{M} for example.
\end{proof}

\begin{lem}[Large-value theorem] \label{lem-largevalue}
Let $N, T \ge 1$ and $V > 0$. Let $F(s) = \sum_{n \le N} a_n n^{-s}$ be a Dirichlet polynomial and let $G = \sum_{n \le N} |a_n|^2$. Let $\mathcal{T}$ be a set of $1$-spaced points $t_r \in [-T, T]$ such that $|F(i t_r)| \ge V$. Then
\[
|\mathcal{T}| \ll (G N V^{-2} + T \min \{ G V^{-2}, G^3 N V^{-6} \} ) (\log 2 NT)^6.
\]
\end{lem}
\begin{proof}
This follows from the mean-value theorem and Huxley's large value theorem. See \cite[Theorem 9.7 and Corollary 9.9]{IK} for example.
\end{proof}

\begin{lem}[Fourth moment estimate] \label{lem-fourthmoment}
Let $T \ge 2$. Then
\[
\int_{-T}^{T} | \zeta( \frac{1}{2} + i t ) |^4 dt \ll T \log^4 T.
\]
\end{lem}
\begin{proof}
See \cite[formula (7.6.1)]{Tit} for example.
\end{proof}

\begin{lem}[Subconvexity bounds] \label{lem-subconvex}
For any $\epsilon > 0$ and $|t| \ge 2$,
\[
\zeta(\frac{1}{2} + i t) \ll_\epsilon |t|^{13 / 84 + \epsilon} \ll |t|^{1/6} \log^2 |t|.
\]
\end{lem}
\begin{proof}
See \cite[formula (8.22)]{IK} for the standard subconvexity bound $|t|^{1/6} \log^2 |t|$. See \cite{B} for the current best subconvexity bound $|t|^{13 / 84 + \epsilon}$ which came from the exponent pair $(13/84 + \epsilon, 55/84 + \epsilon)$.
\end{proof}

\begin{lem}[Van der Corput Process B] \label{lem-processB}
Let $A > 0$ be an absolute constant. Suppose $f(x)$ is a real-valued function such that $0 < \lambda_2 \le f''(x) \le A \lambda$ for all $x \in [a, b]$, and suppose that $|f^{(3)}(x)| \le \frac{A \lambda_2}{b - a}$ and $|f^{(4)}(x)| \le \frac{A \lambda_2}{(b - a)^2}$ throughout this interval. Put $f'(a) = \alpha$ and $f'(b) = \beta$. For integers $\nu \in [\alpha, \beta]$, let $x_\nu$ be the root of the equation $f'(x) = \nu$. Then
\[
\sum_{a \le n \le b} e(f(n)) = e\Bigl(\frac{1}{8} \Bigr) \sum_{\alpha \le \nu \le \beta} \frac{e(f(x_\nu) - \nu x_\nu)}{\sqrt{f''(x_\nu)}} + O(\log(2 + \beta - \alpha) + \lambda_2^{-1/2})
\]
where $e(u) = e^{2 \pi i u}$.
\end{lem}
\begin{proof}
This is Theorem 10 in Chapter 3 of \cite{M} for example.
\end{proof}

\bigskip

Let
\[
M(s) := \sum_{B < b \le 2B} \frac{1}{b^{s}} \; \; \text{ and } \; \; F_{t, B}(u) := \sum_{B \le b \le u} e \Bigl( - \frac{ t \log b }{2 \pi} \Bigr).
\]
\begin{lem} \label{lem-M}
For any $\epsilon > 0$ and $|t| \ge 2$,
\[
M \bigl( \frac{3}{4} + i t \bigr) \ll_\epsilon \frac{|t|^{97/84 + \epsilon}}{B^{9/4}} + \frac{B^{1/4}}{|t|^{1/2}} + \frac{\log |t|}{B^{3/4}}.
\]
\end{lem}
\begin{proof}
Without loss of generality, we may assume that $t \ge 2$ since the negative case is just the complex conjugate of the positive case with the same absolute values. By partial summation,
\[
M \bigl( \frac{3}{4} + i t \bigr) = \int_{B}^{2B} \frac{1}{u^{3/4}} d F_{t, B}(u) = \frac{F_{t, B}(2 B)} {(2 B)^{3/4}} - \frac{F_{t, B}(B)} {B^{3/4}} + \frac{3}{4} \int_{B}^{2 B} \frac{F_{t, B}(u)}{u^{7/4}} du \ll \frac{1}{B^{3/4}} \max_{B \le u \le 2 B} | F_{t, B}(u) |.
\]
By Lemma \ref{lem-processB} with $f(x) = - \frac{t \log x}{2 \pi}$, $\alpha = - \frac{t}{2 \pi B}$, $\beta = -\frac{t}{2 \pi u}$, $x_m = \frac{t}{2 \pi m}$, $a = B$, $b = 2B$, $\lambda_2 = \frac{t}{8 \pi x^2}$, and $A = 100$, we have
\begin{align*}
F_{t, B} (u) =& \sum_{\alpha \le m \le \beta} f''(m)^{-1/2} e\Bigl( f(x_m) - m x_m + \frac{1}{8} \Bigr) + O \Bigl(\frac{B}{|t|^{1/2}} + \log |t| \Bigr) \\
\ll& \frac{1}{\sqrt{|t|}} \Big|\sum_{- \beta \le m \le - \alpha} m e \Bigl(- \frac{t \log m}{2 \pi} \Bigr) \Big| + \frac{B}{|t|^{1/2}} + \log |t| \\
\ll& \frac{\sqrt{|t|}}{B} \max_{t / (4 \pi B) \le v \le t / (2 \pi B)} \Big| \sum_{v \le e \le t / (2 \pi B)} e \Bigl(- \frac{t \log m}{2 \pi} \Bigr) \Big| + \frac{B}{|t|^{1/2}} + \log |t|
\end{align*}
by partial summation. Applying the theory of exponent pairs (see \cite[Chapter 3]{M} for example), the above gives
\[
F_{t, B}(u) \ll \frac{\sqrt{|t|}}{B} \Bigl( \frac{|t|}{|t| / B} \Bigr)^p \cdot \Bigl( \frac{|t|}{B} \Bigr)^q + \frac{B}{|t|^{1/2}} + \log |t| \ll \frac{|t|^{q + 1/2}}{B^{1 + q - p}} + \frac{B}{|t|^{1/2}} + \log |t|
\]
and
\[
M \bigl( \frac{3}{4} + i t \bigr) \ll \frac{|t|^{q + 1/2}}{B^{7/4 + q - p}} + \frac{B^{1/4}}{|t|^{1/2}} + \frac{\log |t|}{B^{3/4}}.
\]
Finally, we have the lemma after inputting Bourgain's exponent pair $p = 13/84 + \epsilon$ and $q = 55/84 + \epsilon$.
\end{proof}


\section{Proof of Theorem \ref{thm2}: Initial Manipulations}

By \eqref{squarefull}, we can rewrite
\begin{equation} \label{z1}
\frac{1}{X} \int_{X}^{2X} \Big| Q(x+y) - Q(x) - \frac{\zeta(3/2)}{\zeta(3)} H \Big|^2 dx = \frac{1}{X} \int_{X}^{2X} \Big| \sum_{x < a^2 b^3 \le x + y} \mu^2(b) - \frac{\zeta(3/2)}{\zeta(3)} H \Big|^2 dx 
\end{equation}
where $y = 2 \sqrt{x} H + H^2$ with $H \le X^{1/2 - \epsilon}$. Observe that
\begin{equation} \label{zeta}
\sum_{b = 1}^{\infty} \frac{\mu^2(b)}{b^s} = \prod_{p} \Bigl(1 + \frac{1}{p^s} \Bigr) = \prod_{p} \frac{(1 - \frac{1}{p^{2s}})}{(1 - \frac{1}{p^s})} = \frac{\zeta(s)}{\zeta(2s)}
\end{equation}
for $\Re{s} > 1$. With some parameter $\lambda > 0$, we will consider those squarefull numbers with $b \le X^{1/3} / H^{\lambda}$ and $b > X^{1/3} / H^{\lambda}$ separately. The right hand side of \eqref{z1} can be rewritten as
\begin{align} \label{z2}
\frac{1}{X} \int_{X}^{2X} & \bigg| \mathop{\sum_{x < a^2 b^3 \le x + y}}_{b \le X^{1/3} / H^{\lambda}} \mu^2(b) - H \sum_{b \le X^{1/3} / H^{\lambda}} \frac{\mu^2(b)}{b^{3/2}} + \mathop{\sum_{x < a^2 b^3 \le x + y}}_{b > X^{1/3} / H^{\lambda}} \mu^2(b) - H \sum_{b > X^{1/3} / H^{\lambda}} \frac{\mu^2(b)}{b^{3/2}} \bigg|^2 dx \nonumber \\
&= I_1 + I_2 + O(I_1^{1/2} I_2^{1/2})
\end{align}
by Cauchy-Schwarz inequality where
\begin{equation} \label{z-bsmall}
I_1 := \frac{1}{X} \int_{X}^{2X} \bigg| \mathop{\sum_{x < a^2 b^3 \le x + y}}_{b \le X^{1/3} / H^{\lambda}} \mu^2(b) - H \sum_{b \le X^{1/3} / H^{\lambda}} \frac{\mu^2(b)}{b^{3/2}} \bigg|^2 dx
\end{equation}
and
\begin{equation} \label{z-blarge}
I_2 := \frac{1}{X} \int_{X}^{2X} \bigg| \mathop{\sum_{x < a^2 b^3 \le x + y}}_{b > X^{1/3} / H^{\lambda}} \mu^2(b) - H \sum_{b > X^{1/3} / H^{\lambda}} \frac{\mu^2(b)}{b^{3/2}} \bigg|^2 dx.
\end{equation}
Observe that when $b > X^{1/3} / H^{\lambda}$, we have $a \le 2 H^{3\lambda/2}$. For fixed $a$, the variable $b$ satisfies
\[
\sqrt[3]{\frac{x}{a^2}} < b \le \sqrt[3]{\frac{x + y}{a^2}} \; \; \text{ or } \; \; \sqrt[3]{\frac{x}{a^2}} < b \le \sqrt[3]{\frac{x}{a^2}} \Bigl(1 + O \Bigl( \frac{y}{x} \Bigr) \Bigr) = \sqrt[3]{\frac{x}{a^2}} + O \Bigl( \frac{H}{X^{1/6} a^{2/3}} \Bigr)
\]
Hence,
\[
\mathop{\sum_{x < a^2 b^3 \le x + y}}_{b > X^{1/3} / H^{\lambda}} \mu^2(b) \ll \sum_{a \le 2 H^{3 \lambda/2}} \Bigl( 1 + \frac{H}{X^{1/6} a^{2/3}} \Bigr) \ll H^{3\lambda/2} + \frac{H^{1 + \lambda/2}}{X^{1/6}},
\]
and
\[
H \sum_{b > X^{1/3} / H^{\lambda}} \frac{\mu^2(b)}{b^{3/2}} \ll \frac{H^{1 + \lambda/2}}{X^{1/6}}.
\]
Putting these two upper bounds into \eqref{z-blarge}, we have
\begin{equation} \label{z-blarge2}
I_2 \ll H^{3 \lambda} + \frac{H^{2 + \lambda}}{X^{1/3}} \ll H^{2/3 - \epsilon}
\end{equation}
provided $\lambda < 2/9 - \epsilon / 3$ and $H \le X^{3/14 - \epsilon}$. To study $I_1$, we separate the range of $b$ between $b \le H^{2/3 + \epsilon}$ and $b > H^{2/3 + \epsilon}$. Then,
\begin{align} \label{JJ}
I_1 =& \frac{1}{X} \int_{X}^{2X} \bigg| \mathop{\sum_{x < a^2 b^3 \le x + y}}_{b \le H^{2/3 + \epsilon}} \mu^2(b) - H \sum_{b \le H^{2/3 + \epsilon}} \frac{\mu^2(b)}{b^{3/2}} \nonumber \\
&+ \mathop{\sum_{x < a^2 b^3 \le x + y}}_{H^{2/3 + \epsilon} < b \le X^{1/3} / H^{\lambda}} \mu^2(b)  - H \sum_{H^{2/3 + \epsilon} < b \le X^{1/3} / H^{\lambda}} \frac{\mu^2(b)}{b^{3/2}} \bigg|^2 dx \nonumber \\
=& J_1 + J_2 + O ( J_1^{1/2} J_2^{1/2} )
\end{align}
where
\begin{equation} \label{J1}
J_1 := \frac{1}{X} \int_{X}^{2X} \bigg| \mathop{\sum_{x < a^2 b^3 \le x + y}}_{b \le H^{2/3 + \epsilon}} \mu^2(b) - H \sum_{b \le H^{2/3 + \epsilon}} \frac{\mu^2(b)}{b^{3/2}} \biggr|^2 dx,
\end{equation}
and
\begin{equation} \label{J2}
J_2 := \frac{1}{X} \int_{X}^{2X} \bigg| \mathop{\sum_{x < a^2 b^3 \le x + y}}_{H^{2/3 + \epsilon} < b \le X^{1/3} / H^{\lambda}} \mu^2(b) - H \sum_{H^{2/3 + \epsilon} < b \le X^{1/3} / H^{\lambda}} \frac{\mu^2(b)}{b^{3/2}} \biggr|^2 dx.
\end{equation}

\section{Proof of Proposition \ref{prop1}}

We are going to study a smoothed version of $J_1$ as defined in \eqref{J1}. Let $L = X^{1 - \epsilon/2}$ and $K$ be some sufficiently large integer (in terms of $\epsilon$). Consider
\begin{equation} \label{z-smallsmooth}
\frac{1}{X} \int_{-\infty}^{\infty} \sigma_{K, X, L}^{\pm} (x) \bigg| \mathop{\sum_{x < a^2 b^3 \le x + y}}_{b \le H^{2/3 + \epsilon}} \mu^2(b) - H \sum_{b \le H^{2/3 + \epsilon}} \frac{\mu^2(b)}{b^{3/2}} \bigg|^2 dx
\end{equation}
We have
\begin{align} \label{ab-sum}
\mathop{\sum_{x < a^2 b^3 \le x + y}}_{b \le H^{2/3 + \epsilon}} \mu^2(b) =& \sum_{b \le H^{2/3 + \epsilon}} \mu^2(b) \sum_{\sqrt{x / b^3} < a \le \sqrt{(x + y) / b^3}} 1 \nonumber \\
=& \sum_{b \le H^{2/3 + \epsilon}} \mu^2(b) \Bigl[ \sqrt{\frac{x + y}{b^3}} - \sqrt{\frac{x}{b^3}} - \psi \Bigl(\sqrt{\frac{x + y}{b^3}} \Bigr) + \psi \Bigl(\sqrt{\frac{x}{b^3}} \Bigr) \Bigr] \nonumber \\
=& H \sum_{b \le H^{2/3 + \epsilon}} \frac{\mu^2(b)}{b^{3/2}} - \sum_{b \le H^{2/3 + \epsilon}} \mu^2(b) \Bigl[ \psi \Bigl(\sqrt{\frac{x + y}{b^3}} \Bigr) - \psi \Bigl(\sqrt{\frac{x}{b^3}} \Bigr) \Bigr] 
\end{align}
where $\psi(u) = u - [u] - 1/2$. For $\psi(u)$, we apply the truncated Fourier expansion
\begin{equation} \label{frac-fourier}
\psi(u) = - \frac{1}{2 \pi i} \sum_{0 < |n| \le N} \frac{1}{n} e(n u) + O \Bigl( \min \bigl(1, \frac{1}{N \| u \|} \bigr) \Bigr).
\end{equation}
We take $N = X^{10}$ and put \eqref{frac-fourier} into \eqref{ab-sum}. The arising error term is $O(1 / X^4)$ unless $\| \sqrt{(x + y) / b^3} \| < X^{-5}$ or $\| \sqrt{x / b^3} \| < X^{-5}$. When $\| \sqrt{(x + y) / b^3} \| < X^{-5}$ or $\| \sqrt{x / b^3} \| < X^{-5}$, we have $|\sqrt{(x + y) / b^3} - a_1| < X^{-5}$ or $| \sqrt{x / b^3} - a_2| < X^{-5}$ for some integers $a_1, a_2$ which implies
\begin{equation} \label{support}
\Big| \frac{x+y}{b^3} - a_1^2 \Bigr|, \; \Big| \frac{x}{b^3} - a_2^2 \Bigr| \ll \frac{1}{X^{4.5}} \; \; \text{ or } \; \; |x+y - a_1^2 b^3|, \; |x - a_2^2 b^3| \ll \frac{1}{X^{3.5}}
\end{equation}
as $b \le H^{2/3 + \epsilon} \le X^{1/3}$. Let $E(x)$ be the error in the second sum in \eqref{ab-sum} coming from the error term of \eqref{frac-fourier} when $\| \sqrt{(x + y) / b^3} \| < X^{-5}$ or $\| \sqrt{x / b^3} \| < X^{-5}$ for some $b \le H^{2/3 + \epsilon}$. As there are $O(\sqrt{X})$ squarefull numbers in the interval $[X, 3X]$ by \eqref{fullcount},
\begin{equation} \label{E}
E(x) \ll H^{2/3 + \epsilon} \ll X^{1/3}, \; \; \text{ and } \; \; E(x) \text{ is supported on } O(\sqrt{X}) \text{ intervals of length } O\Bigl(\frac{1}{X^{3.5}} \Bigr) \text{ each}
\end{equation}
by \eqref{support}. Thus, we can replace \eqref{z-smallsmooth} by
\begin{align} \label{z-smallsmooth2}
\frac{1}{4 \pi^2 X} & \int_{-\infty}^{\infty} \sigma_{K, X, L}^{\pm} (x) \bigg| \sum_{b \le H^{2/3 + \epsilon}} \mu^2(b) \sum_{0 < |n| \le N} \frac{1}{n} e \Bigl(\frac{n \sqrt{x}}{b^{3/2}} \Bigl) \Bigl[1 - e \Bigl( \frac{n H}{b^{3/2}} \Bigr) \Bigr] + E(x) + O\Bigl(\frac{1}{X^4} \Bigr) \bigg|^2 dx \nonumber \\
=& \frac{1}{X} \Bigl[J_{\pm} + O \Bigl(\frac{J_{\pm}^{1/2}}{X} + \frac{1}{X^{2}}\Bigr) \Bigr]
\end{align}
by Cauchy-Schwarz inequality and $\int_{-\infty}^{\infty} \sigma_{K, X, L}^{\pm} (x) |E(x) + O(1/X^4)|^2 dx \ll X^{-2}$ via \eqref{E}. Here
\begin{equation} \label{z-smallsmooth2.5}
J_{\pm} := \frac{1}{4 \pi^2} \int_{-\infty}^{\infty} \sigma_{K, X, L}^{\pm} (x) \bigg| \sum_{b \le H^{2/3 + \epsilon}} \mu^2(b) \sum_{0 < |n| \le N} \frac{1}{n} e \Bigl(\frac{n \sqrt{x}}{b^{3/2}} \Bigl) \Bigl[1 - e \Bigl( \frac{n H}{b^{3/2}} \Bigr) \Bigr] \bigg|^2 dx.
\end{equation}
Note that, since the characteristic function of the interval $[X, 2X]$ is between $\sigma_{K,X,L}^-$ and $\sigma_{K,X,L}^+$, we have
\begin{equation} \label{newintermediate}
\frac{1}{X} \Bigl[J_{-} + O \Bigl(\frac{J_{-}^{1/2}}{X} + \frac{1}{X^{2}}\Bigr) \Bigr] \le J_1 \le \frac{1}{X} \Bigl[J_{+} + O \Bigl(\frac{J_{+}^{1/2}}{X} + \frac{1}{X^{2}}\Bigr) \Bigr].
\end{equation}
Expanding \eqref{z-smallsmooth2.5} out,
\begin{align} \label{z-smallsmooth3}
J_{\pm} =\frac{1}{4 \pi^2} & \sum_{b_1, b_2 \le H^{2/3 + \epsilon}} \sum_{0 < |n_1|, |n_2| \le N} \mu^2(b_1) \mu^2(b_2) \frac{1}{n_1 n_2} \nonumber \\
&\times \Bigl[1 - e \Bigl( \frac{n_1 H}{b_1^{3/2}} \Bigr) \Bigr] \overline{\Bigl[1 - e \Bigl( \frac{n_2 H}{b_2^{3/2}} \Bigr) \Bigr]} \int_{-\infty}^{\infty} \sigma_{K, X, L}^{\pm} (x) e \Bigl(x^{1/2} \Bigl(\frac{n_1}{b_1^{3/2}} - \frac{n_2}{b_2^{3/2}} \Bigr) \Bigr) dx.
\end{align}
By Lemma \ref{lem1} with $H = X^\epsilon$, $W = 1$, $Q = X$, $L = X^{1 - \epsilon/2}$ and $R = 1/X^{1-\epsilon}$, the contribution from those terms with $| \frac{n_1}{b_1^{3/2}} - \frac{n_2}{b_2^{3/2}} | > \frac{1}{X^{1/2 - \epsilon}}$ is
\begin{equation} \label{off}
\ll X^{- K \epsilon / 2} \ll \frac{1}{X^5}
\end{equation}
by picking $K = [20 / \epsilon]$ for example. Hence, we may restrict the sum in \eqref{z-smallsmooth3} to those integers for which
\begin{equation*}
\bigg| \frac{n_1}{b_1^{3/2}} - \frac{n_2}{b_2^{3/2}} \bigg| \le \frac{1}{X^{1/2 - \epsilon}}.
\end{equation*}
We separate those $(n_1, n_2, b_1, b_2)$ for which $\frac{n_1}{b_1^{3/2}} = \frac{n_2}{b_2^{3/2}}$ (i.e. $n_1^2 b_2^3 = n_2^2 b_1^3$) and those for which $\frac{n_1}{b_1^{3/2}} \neq \frac{n_2}{b_2^{3/2}}$. In the first case, we must have $b_1 = b_2 = b$ and $n_1 = n_2 = n$ which contributes
\begin{align*}
\frac{\overline{\sigma_{K,X,L}^{\pm}}}{4 \pi^2} \sum_{b \le H^{2/3 + \epsilon}} \mu^2(b) \sum_{n \neq 0} \frac{1}{n^2}
\Big|1 - e \Bigl( \frac{n H}{b^{3/2}} \Bigr) \Big|^2 + O \Bigl(\frac{1}{X^5}\Bigr)
\end{align*}
where the error term comes from adding $|n| > N = X^{10}$ and
\begin{equation} \label{mean-sigma}
\overline{\sigma_{K,X,L}^{\pm}} = \int_{-\infty}^{\infty} \sigma_{K,X,L}^{\pm}(x) dx = X + O_\epsilon(X^{1- \epsilon/2}).
\end{equation}
As
\[
\Big|1 - e \Bigl( \frac{n H}{b^{3/2}} \Bigr) \Big| = 2 \Big| \sin \Bigl( \frac{n \pi H}{b^{3/2}} \Bigr) \Big|,
\]
the diagonal terms ($\frac{n_1}{b_1^{3/2}} = \frac{n_2}{b_2^{3/2}}$) from \eqref{z-smallsmooth3} contribute 
\begin{equation} \label{diagonal}
(1 + O(X^{-\epsilon /2})) X H^2  \sum_{b \le H^{2/3 + \epsilon}} \frac{\mu^2(b)}{b^3} \sum_{n \neq 0} S\Bigl( \frac{n H}{b^{3/2}} \Bigr)^2 + O \Bigl(\frac{1}{X^5}\Bigr)
\end{equation}
where $S(x) = \frac{\sin \pi x}{\pi x}$. Splitting $n_i$ and $b_i$ dyadically, it remains to show that the contribution from those non-diagonal terms with $n_1^2 b_2^3 \neq n_2^2 b_1^3$ satisfies
\begin{equation} \label{bound2}
\min \Bigl(\frac{1}{N_1}, \frac{H}{B_1^{3/2}} \Bigr) \min \Bigl(\frac{1}{N_2}, \frac{H}{B_2^{3/2}} \Bigr) \# \Bigl\{ (b_1, b_2, n_1, n_2) : \begin{matrix} b_i \sim B_i, \\ n_i \sim N_i, \end{matrix} 0 < \Big| \frac{n_1}{b_1^{3/2}} - \frac{n_2}{b_2^{3/2}} \Big| \le \frac{1}{X^{1/2 - \epsilon}} \Bigr\} \ll_\epsilon H^{2/3 - \epsilon}
\end{equation}
for any $B_1, B_2 \le H^{2/3 + \epsilon}$ and $N_1, N_2 \le N$. We have $\frac{B_1^{3/2}}{X^{1/2 - \epsilon}}, \frac{B_2^{3/2}}{X^{1/2 - \epsilon}} \ll \frac{1}{X^{\epsilon / 4}} < 0.1$ as $B_1, B_2 \ll X^{1/3 - 5\epsilon/6}$. So, we can assume that $N_1 \ge \frac{10 B_1^{3/2}}{X^{1/2 - \epsilon}}$ and $N_2 \ge \frac{10 B_2^{3/2}}{X^{1/2 - \epsilon}}$. Thus,
\begin{equation} \label{dioph-ineq}
\Big| \frac{n_1}{b_1^{3/2}} - \frac{n_2}{b_2^{3/2}} \Big| \le \frac{1}{X^{1/2 - \epsilon}} \; \; \text{ or } \; \; |n_1^2 b_2^3 - n_2^2 b_1^3| \ll \frac{N_1 B_1^{3/2} B_2^{3}}{X^{1/2 - \epsilon}} \asymp \frac{N_1^{1/2} N_2^{1/2} B_1^{9/4} B_2^{9/4}}{X^{1/2 - \epsilon}}
\end{equation}
has no solution unless $N_1 B_2^{3/2} \asymp N_2 B_1^{3/2}$. Let $n = \gcd(n_1, n_2)$ and $b = \gcd(b_1, b_2)$. We can rewrite
\[
n_1 = n n_1', \; \; n_2 = n n_2', \; \; b_1 = b b_1', \; \; b_2 = b b_2'.
\]
Inequality \eqref{dioph-ineq} implies
\begin{equation*}
\Big| \frac{n_1'}{n_2'} - \frac{b_1'^{3/2}}{b_2'^{3/2}} \Big| \ll \frac{B_1^{3/2}}{N_2 X^{1/2 - \epsilon}}.
\end{equation*}
Note that $|\frac{n_1'}{n_2'} - \frac{n_1''}{n_2''}| \gg \frac{1}{(N_2 / n)^2}$ for distinct fractions. Hence, for fixed $b$, $n$, $b_1'$ and $b_2'$, the number of $n_1'$ and $n_2'$ satisfying \eqref{dioph-ineq} is
\[
\ll \frac{B_1^{3/2} / (N_2 X^{1/2 - \epsilon})}{1 / (N_2 / n)^2} + 1 = \frac{B_1^{3/2} N_2}{n^2 X^{1/2 - \epsilon}} + 1.
\]
Thus, the quantity in \eqref{bound2} is
\begin{align} \label{bound3}
&\ll \sum_{b \le 2 \min(B_1, B_2)} \sum_{n \le 2 \min(N_1, N_2)} \sum_{b_1' \ll B_1 / b} \; \sum_{b_2' \ll B_2 / b} \min \Bigl(\frac{1}{N_1}, \frac{H}{B_1^{3/2}} \Bigr) \min \Bigl(\frac{1}{N_2}, \frac{H}{B_2^{3/2}} \Bigr) \Bigl( \frac{B_1^{3/2} N_2}{n^2 X^{1/2 - \epsilon}} + 1 \Bigr) \nonumber \\
&\ll \sum_{b} \sum_{n} \sum_{b_1' \ll B_1 / b} \; \sum_{b_2' \ll B_2 / b} \frac{1}{N_1} \frac{1}{N_2} \frac{B_1^{3/2} N_2}{n^2 X^{1/2 - \epsilon}} + \sum_{b} \sum_{n \le 2 \min(N_1, N_2)} \sum_{b_1' \ll B_1 / b} \sum_{b_2' \ll B_2 / b} \frac{1}{N_1} \frac{1}{N_2} \nonumber \\
&\ll \frac{B_1^{5/2} B_2}{X^{1/2 - \epsilon}} + \sum_{b} \sum_{n \le 2 \sqrt{N_1 N_2}} \frac{B_1 B_2}{b^2 N_1 N_2} \ll \frac{H^{7/3 + 7\epsilon/2}}{X^{1/2 - \epsilon}} + \frac{H^{4/3 + 2 \epsilon}}{\sqrt{N_1 N_2}} \ll_\epsilon H^{2/3 - \epsilon}
\end{align}
as long as $H \le X^{3/10 - 2 \epsilon}$ and $N_1 N_2 \ge H^{4/3 + 6 \epsilon}$. 

\bigskip

It remains to deal with $N_1 N_2 < H^{4/3 + 6 \epsilon}$. Since $n^2 b^3$ divides both $n_1^2 b_2^3$ and $n_2^2 b_1^3$, and $n_1^2 b_2^3 \neq n_2^2 b_1^3$, inequality \eqref{dioph-ineq} implies
\[
b^3 \le n^2 b^3 \le |n_1^2 b_2^3 - n_2^2 b_1^3| \ll \frac{N_1^{1/2} N_2^{1/2} B_1^{9/4} B_2^{9/4}}{X^{1/2 - \epsilon}}
\]
and
\[
b \ll \frac{(N_1 N_2)^{1/6} (B_1 B_2)^{3/4}}{X^{1/6 - \epsilon/3}} \ll \frac{H^{2/9 + \epsilon} \cdot H^{1 + 3\epsilon/2}}{X^{1/6 - \epsilon / 3}} \ll H^{2/3 - \epsilon}
\]
as long as $H \le X^{3/10 - 3 \epsilon}$. Then, inequality \eqref{dioph-ineq} implies
\[
\Big| \frac{b_1'}{b_2'} - \frac{n_1'^{2/3}}{n_2'^{2/3}} \Big| \ll \frac{B_1 B_2^{1/2}}{N_2 X^{1/2 - \epsilon}}.
\]
Note that $|\frac{b_1'}{b_2'} - \frac{b_1''}{b_2''}| \gg \frac{1}{(B_2 / b)^2}$ for distinct fractions. Hence, for fixed $b$, $n$, $n_1'$ and $n_2'$, the number of $b_1'$ and $b_2'$ satisfying \eqref{dioph-ineq} is
\[
\ll \frac{B_1 B_2^{1/2} / (N_2 X^{1/2 - \epsilon})}{1 / (B_2 / b)^2} + 1 = \frac{B_1 B_2^{5/2}}{b^2 N_2 X^{1/2 - \epsilon}} + 1.
\]
Thus, the quantity in \eqref{bound2} is
\begin{align} \label{bound4}
&\ll \sum_{b \le 2 \min(B_1, B_2)} \sum_{n \le 2 \min(N_1, N_2)} \sum_{n_1' \ll N_1 / n} \; \sum_{n_2' \ll N_2 / n} \min \Bigl(\frac{1}{N_1}, \frac{H}{B_1^{3/2}} \Bigr) \min \Bigl(\frac{1}{N_2}, \frac{H}{B_2^{3/2}} \Bigr) \Bigl( \frac{B_1 B_2^{5/2}}{b^2 N_2 X^{1/2 - \epsilon}} + 1 \Bigr) \nonumber \\
&\ll \sum_{b} \sum_{n} \sum_{n_1'} \sum_{n_2'} \frac{1}{N_1 N_2} \cdot \frac{B_1 B_2^{5/2}}{b^2 N_2 X^{1/2 - \epsilon}} + \sum_{b \ll H^{2/3 - \epsilon}} \sum_{n} \sum_{n_1'} \sum_{n_2'} \frac{1}{N_1 N_2} \ll \frac{B_1 B_2^{5/2}}{X^{1/2 - \epsilon}} + H^{2/3 - \epsilon} \ll_\epsilon H^{2/3 - \epsilon}
\end{align}
as long as $H \le X^{3/10 - 3\epsilon}$. Combining \eqref{z-smallsmooth2.5}, \eqref{newintermediate}, \eqref{z-smallsmooth3}, \eqref{off}, \eqref{mean-sigma}, \eqref{diagonal}, \eqref{bound3} and \eqref{bound4}, we have
\begin{equation} \label{Jfinal}
J_1 = (1 + O(X^{-\epsilon /2})) 2 H^2  \sum_{b \le H^{2/3 + \epsilon}} \frac{\mu^2(b)}{b^3} \sum_{n \ge 1} S\Bigl( \frac{n H}{b^{3/2}} \Bigr)^2 + O_\epsilon ( H^{2/3 - \epsilon / 2} )
\end{equation}
when $X^\epsilon \le H \le X^{3/10 - 3\epsilon}$. 

\section{Proof of Proposition \ref{prop2}} \label{sec-V}

Throughout this section, we assume that $H \le X^{9/46 - 3 \epsilon}$. Recall from \eqref{J2}, we want to establish that
\begin{equation} \label{J2final}
J_2 = \frac{1}{X} \int_{X}^{2X} \bigg| \mathop{\sum_{x < a^2 b^3 \le x + y}}_{H^{2/3 + \epsilon} < b \le X^{1/3} / H^{\lambda}} \mu^2(b) - H \sum_{H^{2/3 + \epsilon} < b \le X^{1/3} / H^{\lambda}} \frac{\mu^2(b)}{b^{3/2}} \biggr|^2 dx \ll H^{2/3 - \epsilon / 8}.
\end{equation}
We will follow the proof of Proposition 2 in \cite{GMRR} closely. First, one can split the sums into dyadic ranges according to the size of $B < b \le 2 B$. It suffices to show that, for $B \in [H^{2/3 + \epsilon}, 2X^{1/3} / H^\lambda]$, we have
\[
J_{2, B} := \frac{1}{X} \int_{X}^{2X} \bigg| \mathop{\sum_{x < a^2 b^3 \le x + y}}_{B < b \le 2B} \mu^2(b) - H \sum_{B < b \le 2B} \frac{\mu^2(b)}{b^{3/2}} \biggr|^2 dx \ll H^{2/3 - \epsilon / 4}.
\]
Define
\[
A(x) := \mathop{\sum_{a^2 b^3 \le x}}_{B < b \le 2B} \mu^2(b) - \sqrt{x} \sum_{B < b \le 2B} \frac{\mu^2(b)}{b^{3/2}}, \; \; \text{ and } \; \; B(x) := A(x^2).
\]
Since $y = 2 \sqrt{x} H + H^2$, one can show that
\begin{equation*}
J_{2, B} = \frac{1}{X} \int_{X}^{2X} \big| A \bigl( (\sqrt{x} + H)^2 \bigr) - A \bigl( (\sqrt{x})^2 \bigr) \big|^2 dx \ll \frac{1}{\sqrt{X}} \int_{\sqrt{X}}^{\sqrt{2X}} | B(u + H) - B(u) |^2 du
\end{equation*}
by a change of variable $u = \sqrt{x}$. Following \cite{GMRR} by using a technique by Saffari and Vaughan, we have
\begin{equation} \label{J2B}
J_{2, B} \ll \frac{1}{\sqrt{X}} \int_{\sqrt{X}}^{3 \sqrt{X}} | B(u (1 + \theta)) - B(u) |^2 du
\end{equation}
for some $\theta \in [ \frac{H}{3 \sqrt{X}} , \frac{3H}{\sqrt{X}} ]$. Using contour integration,
\begin{equation} \label{BB}
B(e^y) = A(e^{2y}) = \frac{1}{2 \pi i} \int_{1 - i \infty}^{1 + i \infty} \frac{e^{2 y s}}{s} \zeta(2s) M(3s) \, ds - e^{y} \sum_{B < b \le 2B} \frac{\mu^2(b)}{b^{3/2}},
\end{equation}
where
\[
M(s) := \sum_{B < b \le 2B} \frac{\mu^2(b)}{b^s}.
\]
Moving the contour to the line $\Re s = 1/4$ and noticing that the residue from $s = 1/2$ cancels with the second term on the right hand side of \eqref{BB}, we have
\[
\frac{B(e^{w + x}) - B(e^x)}{e^{x/2}} = \frac{1}{2 \pi} \int_{-\infty}^{\infty} \frac{e^{2 w( 1/4 + i t)} - 1}{1/4 + i t} \, e^{2 i t x} \, \zeta(\frac{1}{2} + 2 i t) M(\frac{3}{4} + 3 i t) \, dt
\]
where $e^w = 1 + \theta$ with $w \asymp H / \sqrt{X}$. Hence, by Plancherel's identity,
\begin{equation} \label{BPlan}
\int_{0}^{\infty} |B(e^{x+w}) - B(e^x)|^2 \, \frac{dx}{e^x} \ll \int_{-\infty}^{\infty} \Big| \frac{e^{2 w( 1/4 + i t)} - 1}{1/4 + i t} \Big|^2 \cdot \big|\zeta(\frac{1}{2} + 2 i t) M(\frac{3}{4} + 3 i t) \big|^2 \, dt.
\end{equation}
Combining \eqref{J2B} and \eqref{BPlan} with a change of variable $u = e^x$, we have
\begin{align} \label{J2Bnew}
J_{2,B} \ll& \sqrt{X} \int_{0}^{\infty} |B(u(1 + \theta)) - B(u)|^2 \frac{du}{u^2} \nonumber \\
\ll& \sqrt{X} \int_{-\infty}^{\infty} \Big| \frac{e^{2 w( 1/4 + i t)} - 1}{1/4 + i t} \Big|^2 \big|\zeta(\frac{1}{2} + 2 i t) M(\frac{3}{4} + 3 i t) \big|^2 \, dt \nonumber \\
\ll& \sqrt{X} \int_{-\infty}^{\infty} \min \Bigl( \frac{ H^2 }{X}, \frac{1}{|t|^2} \Bigr) \cdot \big|\zeta(\frac{1}{2} + 2 i t) M(\frac{3}{4} + 3 i t) \big|^2 \, dt.
\end{align}
By Lemma \ref{lem-subconvex}, the contribution from $|t| \ge X^2$ to the above integral is
\[
\ll \sqrt{X} \int_{X^2}^{\infty} |t|^{-5/3 + \epsilon} B^{1/2} \, dt \ll 1
\]
as $B \ll X^{1/3} / H^\lambda$. The contribution of $|t| \le X^2$ to the right hand side of \eqref{J2Bnew} is
\begin{align*}
\ll& \frac{H^2}{\sqrt{X}} \int_{|t| \le 2\sqrt{X} / H} \big|\zeta(\frac{1}{2} + 2 i t) M(\frac{3}{4} + 3 i t) \big|^2 \, dt \\
&+ \sqrt{X} \int_{\sqrt{X} / H}^{X^2} \frac{1}{T^2} \cdot \frac{1}{T} \int_{T \le |t| \le 2T} \big|\zeta(\frac{1}{2} + 2 i t) M(\frac{3}{4} + 3 i t) \big|^2 \, dt \, dT \\
\ll& H \sup_{\sqrt{X} / H \le T \le X^2} \frac{1}{T} \int_{-T}^{T} \big|\zeta(\frac{1}{2} + 2 i t) M(\frac{3}{4} + 3 i t) \big|^2 \, dt
\end{align*}
Thus,
\begin{equation} \label{tempJ2B}
J_{2, B} \ll H \sup_{\sqrt{X} / H \le T \le X^2} \frac{1}{T} \int_{-T}^{T} \big|\zeta(\frac{1}{2} + 2 i t) M(\frac{3}{4} + 3 i t) \big|^2 \, dt + 1.
\end{equation}

First, note that $B \le T$ for otherwise
\[
\frac{\sqrt{X}}{H} \le T < B \ll \frac{X^{1/3}}{H^{2/9 - \epsilon / 3}}
\]
which implies $X^{1/6} \ll H^{7/9 + \epsilon / 3}$ or $X^{3/14 - \epsilon} \ll H$, a contradiction. By Lemmas \ref{lem-meanvalue} and \ref{lem-subconvex}, the above implies
\begin{align*}
J_{2,B} \ll& H \frac{T^{13/42 + \epsilon/2}}{T} \int_{-T}^{T} \big|M(\frac{3}{4} + 3 i t) \big|^2 \, dt + 1 \\
\ll& H \frac{T^{13/42 + \epsilon/2}}{T} (T + B) \sum_{B < b \le 2 B} \frac{1}{b^{3/2}} \ll H \frac{T^{13/42 + \epsilon/2}}{\sqrt{B}} \ll H^{2/3 - \epsilon/2}
\end{align*}
when $B \ge H^{2/3 + \epsilon} T^{13/21 + \epsilon}$. From now on, we assume that $B < H^{2/3 + \epsilon} T^{13/21 + \epsilon}$.

\bigskip

By Cauchy-Schwarz inequality,
\begin{align} \label{Jbasic}
J_{2, B} \ll& H \sup_{\sqrt{X} / H \le T \le X^2} \Bigl( \frac{1}{T} \int_{-T}^{T} \big|\zeta(\frac{1}{2} + 2 i t) \big|^4 dt \Bigr)^{1/2} \Bigl( \frac{1}{T} \int_{-T}^{T} \big| M(\frac{3}{4} + 3 i t) \big|^4 dt \Bigr)^{1/2} \nonumber \\
\ll& H \log^2 X \sup_{\sqrt{X} / H \le T \le X^2} \Bigl( \frac{1}{T} \int_{-T}^{T} \big| M(\frac{3}{4} + 3 i t) \big|^4 dt \Bigr)^{1/2}
\end{align}
by Lemma \ref{lem-fourthmoment}. Note that
\[
M^2 (\frac{3}{4} + 3 i t) = \sum_{B^2 < b' \le 4B^2} \frac{d_B( b' )}{b'^{3/4}} \; \; \text{ with } \; \; d_B(b') = \# \{(b_1, b_2): b' = b_1 b_2, \; B < b_1, b_2 \le 2B \}.
\]
By Lemma \ref{lem-meanvalue}, \eqref{Jbasic} gives
\[
J_{2,B} \ll H \log^2 X \sup_{\sqrt{X} / H \le T \le X^2} \Bigl( \frac{T + B^2}{T} \sum_{B^2 < b \le 4 B^2} \frac{d_B(b')^2}{b'^{3/2}} \Bigr)^{1/2} + 1 \ll_\epsilon H X^{\epsilon^2 / 4} \Bigl( \frac{1}{\sqrt{B}} + \frac{\sqrt{B}}{\sqrt{T}} \Bigr) + 1.
\]
Since $X^\epsilon \le H$ and $H^{2/3 + \epsilon} \le B$, the above implies $J_{2, B} \ll_\epsilon H^{2/3 - \epsilon / 4}$ when $B \le T / H^{2/3 + \epsilon}$.

\bigskip

From now on, we assume that $B > T / H^{2/3 + \epsilon}$. By Lemma \ref{lem-meanvalue}, the contribution from those $t$'s with $|M(\frac{3}{4} +  3 it)| \le \frac{1}{H^{1/6 + \epsilon}}$ to \eqref{tempJ2B} is also acceptable. Define the set
\[
S(V) := \{ t \in [-T, T] : V \le |M(\frac{3}{4} +  3 it)| < 2 V \} \; \; \text{ for } \; \; \frac{1}{H^{1/6 + \epsilon}} \le V \ll B^{1/4}.
\]
Here the upper bound for $V$ simply comes from $|M(\frac{3}{4} +  3 it)| \le |M(\frac{3}{4})|$. We want to obtain
\begin{equation} \label{J2Bsplit}
\mathcal{I} := \frac{H}{T} V^2 \int_{S(V)} \bigl| \zeta( \frac{1}{2} + 2 i t) \bigr|^2 dt \ll_\epsilon H^{2/3 - \epsilon}
\end{equation}
for all $V$ as that would imply $J_{2, B} \ll H^{2/3 - \epsilon/4}$. By Lemma \ref{lem-largevalue},
\begin{equation} \label{largevalue}
|S(V)| \ll \Bigl( \frac{\sqrt{B}}{V^2} + T \min \Bigl( \frac{1}{\sqrt{B} V^2}, \frac{1}{\sqrt{B} V^6} \Bigr) \Bigr) \log^6 2X.
\end{equation}
If the first term in \eqref{largevalue} dominates, then by Lemma \ref{lem-subconvex} and $B < H^{2/3 + \epsilon} T^{13/21 + \epsilon}$, the contribution to $\mathcal{I}$ is
\[
\ll T^{13/42 + \epsilon/2} \frac{H}{T} V^2 \frac{\sqrt{B}}{V^2} \log^6 2X \ll \frac{H \sqrt{B}}{T^{29/42 - \epsilon}}
\]
which is acceptable when $T \ge H^{14/29 + 2\epsilon} B^{21/29 + 2 \epsilon}$. We claim that this always holds. Suppose the contrary, we have
\[
\frac{\sqrt{X}}{H} \le T < H^{14/29 + 2\epsilon} B^{21/29 + 2 \epsilon} \; \; \text{ or } \; \; \frac{X^{29/42 - 2\epsilon}}{H^{43/21 + 2 \epsilon}} < B.
\]
This together with $B \le X^{1/3} / H^{2/9 - \epsilon / 3}$ imply $H^{15/42 - 2\epsilon} < H^{115 / 63 + 7 \epsilon/3}$ or $H > X^{9/46 - 3 \epsilon}$, a contradiction. Hence, when the first term in \eqref{largevalue} dominates, we have an acceptable error.

\bigskip

Now, suppose the second term in \eqref{largevalue} dominates. By Cauchy-Schwarz inequality and Lemma \ref{lem-fourthmoment}, the contribution to $\mathcal{I}$ is
\[
\ll \frac{H}{T} V^2 |S(V)|^{1/2} \Bigl( \int_{-T}^{T} \big| \zeta \bigl(\frac{1}{2} + it \bigr) \big|^2 dt \Bigr)^{1/2} \ll H V^2 \min \Bigl( \frac{1}{B^{1/4} V} , \frac{1}{B^{1/4} V^3} \Bigr) \ll \frac{H}{B^{1/4}}
\]
which is acceptable when $B \ge H^{4/3 + \epsilon}$. It remains to deal with the case when $H^{2/3 + \epsilon} \le B < H^{4/3 + \epsilon}$ and $B > T / H^{2/3 + \epsilon}$.

\bigskip

Now, we assume that $H \le X^{1 / (6 - 2 \delta) - \epsilon}$ with $0.02 < \delta < 1/3$. We further break $T / H^{2/3 + \epsilon} < B < H^{4/3 + \epsilon}$ and $\sqrt{X} / H \le T$ into different ranges. Note that the lower and upper bounds on $B$ gives $T < H^{2 + 2 \epsilon}$. 

\bigskip

Firstly, consider $\sqrt{X} / H \le T \le H^{2 - \delta + \epsilon}$. This implies $\sqrt{X} \le H^{3 - \delta + \epsilon}$ or $X^{1/(6 - 2 \delta + 2 \epsilon)} \le H$ which contradicts our assumption on $H$. Hence, we must have $T > H^{2 - \delta + \epsilon}$.

\bigskip

Secondly, consider $B \le H^{4/3 - \delta}$. This implies
\[
T / H^{2/3 + \epsilon} < H^{4/3 - \delta} \; \; \text{ or } \; \; T < H^{2 - \delta + \epsilon}
\]
which contradicts the above range for $T$. Therefore, we can restrict our attention to $H^{2 - \delta + \epsilon} < T \le H^{2 + 2 \epsilon}$ and $H^{4/3 - \delta} < B < H^{4/3 + \epsilon}$. Over these ranges, we have
\[
\frac{1}{H^{1/6 + \epsilon}} \le V \le \Big| M \bigl( \frac{3}{4} + i t \bigr) \Big| \ll \frac{H^{97/42 + 4 \epsilon}}{H^{3 - 9 \delta / 4}} + \frac{H^{1/3 + \epsilon}}{H^{(2 - \delta) / 2}} + \frac{\log H}{H^{1 - 3 \delta/4}} \ll H^{9 \delta / 4 - 29 / 42 + 4 \epsilon}
\]
by Lemma \ref{lem-M} and $0.02 < \delta < 1/3$. This situation cannot happen when $\delta = 44/189 - 3 \epsilon$. Summarizing all of the above, we have an acceptable error for $J_{2, B}$ and, hence, Proposition \ref{prop2} when
\[
X^\epsilon \le H \le X^{\frac{1}{6 - 2 (44/189 - 3 \epsilon)} - \epsilon} = X^{\frac{189}{1046} - 2 \epsilon} = X^{0.180688... - 2 \epsilon}.
\]

\section{Finishing the proof of Theorem \ref{thm2}}

Applying Proposition \ref{prop3} to Proposition \ref{prop1}, we have 
\[
J_1 = (1 + O(X^{-\epsilon /2})) H^{2/3} \cdot \frac{4 \zeta(4/3)}{3 \zeta(2)} \int_{0}^{\infty} |S(y)|^2 y^{1/3} dy + O_\epsilon ( H^{2/3 - \epsilon / 6} )
\]
when $X^\epsilon \le H \le X^{3/10 - 3 \epsilon}$. By Proposition \ref{prop2}, we have $J_2 \ll H^{2/3 - \epsilon / 8}$ when $X^\epsilon \le H \le X^{\frac{189}{1046} - 2 \epsilon}$. Putting these and \eqref{z-blarge2} into \eqref{JJ} and \eqref{z2}, we get
\[
\frac{1}{X} \int_{X}^{2X} \Big| Q(x+y) - Q(x) - \frac{\zeta(3/2)}{\zeta(3)} H \Big|^2 dx \\
= (1 + O(X^{-\epsilon /2})) H^{2/3} \cdot \frac{4 \zeta(4/3)}{3 \zeta(2)} \int_{0}^{\infty} |S(y)|^2 y^{1/3} dy + O_\epsilon ( H^{2/3 - \epsilon / 16})
\]
when $X^\epsilon \le H \le X^{\frac{189}{1046} - 2 \epsilon}$.

\section{Proof of Theorem \ref{thm3}}

We can do better assuming the RH or the LH. We set the parameter $\lambda = 0$ here. Then the sum over squarefull numbers $a^2 b^3$ in $I_2$ of \eqref{z-blarge} is empty unless $a = 1$. Hence, $I_2 \ll H^2 / X^{1/3} + 1 \ll H^{2/3 - \epsilon}$ when $H \le X^{1/4 - \epsilon}$. Under either RH or LH, the bound \eqref{tempJ2B} becomes
\begin{equation} \label{tempJ2B1}
J_{2, B} \ll_\epsilon H X^{\epsilon^2 / 4} \sup_{\sqrt{X} / H \le T \le X^2} \frac{1}{T} \int_{-T}^{T} \big|M(\frac{3}{4} + 3 i t) \big|^2 \, dt + 1
\end{equation}
and we do not need to worry about whether $B \le T$ or not. Using the same notation $S(V)$ for large value of Dirichlet polynomial as in Section \ref{sec-V}, we again have the bound \eqref{largevalue}. Suppose the first term in \eqref{largevalue} dominates, then such contribution to $J_{2, B}$ is
\[
\ll_\epsilon H X^{\epsilon^2 / 2} \sup_{\sqrt{X} / H \le T \le X^2} \frac{\sqrt{B}}{T} \ll \frac{H^2 X^{\epsilon^2 / 2} \sqrt{B}}{\sqrt{X}} \ll_\epsilon H^{2/3 - \epsilon / 4}
\]
when $B \ll X^{1 - \epsilon^2} / H^{8/3 + \epsilon}$. Since $B \ll X^{1/3}$, this is true when $H \le X^{1 / 4 - \epsilon}$.

\bigskip

Suppose the second term in \eqref{largevalue} dominates, then such contribution to $J_{2, B}$ is
\[
\ll_\epsilon H X^{\epsilon^2 / 4} V^2 \frac{1}{\sqrt{B}} \min \Bigl( \frac{1}{V^2}, \frac{1}{V^6} \Bigr) \ll \frac{H X^{\epsilon^2 / 2}}{\sqrt{B}} \ll H^{2/3 - \epsilon / 4}
\]
since $H^{2/3 + \epsilon} \le B$ and $H \ge X^{\epsilon}$. Applying all these and Propositions \ref{prop1} and \ref{prop3} to \eqref{JJ} and \eqref{z2}, we get
\[
\frac{1}{X} \int_{X}^{2X} \Big| Q(x+y) - Q(x) - \frac{\zeta(3/2)}{\zeta(3)} H \Big|^2 dx \\
= H^{2/3} \cdot \frac{4 \zeta(4/3)}{3 \zeta(2)} \int_{0}^{\infty} |S(y)|^2 y^{1/3} dy + O_\epsilon ( H^{2/3 - \epsilon / 16})
\]
when $X^\epsilon \le H \le X^{1/4 - \epsilon}$, and hence the theorem.

\section{Proof of Proposition \ref{prop3}}

In this section, we focus on finding an asymptotic formula for the diagonal contribution:
\begin{equation} \label{diagonal-main}
2 H^2 \sum_{b \le H^{2/3 + \epsilon}} \frac{\mu^2(b)}{b^3} \sum_{n \ge 1} S\Bigl( \frac{n H}{b^{3/2}} \Bigr)^2 \; \; \text{ where } \; \; S(x) = \frac{\sin \pi x}{\pi x}.
\end{equation}
As in \cite{GMRR} and \cite{C}, we consider a smoother sum first. Fix $\epsilon > 0$ and $K_0 > 0$. Suppose that $w: \mathbb{R} \rightarrow \mathbb{C}$ satisfies
\begin{equation} \label{deriv}
|w^{(k)}(y)| \le K_0 \frac{H^{l \epsilon / 4}}{(1 + |y|)^l}
\end{equation}
for all $k, l \in \{0, 1, 2, 3, 4 \}$. We claim that
\begin{equation} \label{smoothW}
2 H^2 \sum_{b \le H^{2/3 + \epsilon}} \frac{\mu^2(b)}{b^3} \sum_{n \ge 1} \Big| w\Bigl( \frac{n H}{b^{3/2}} \Bigr) \Big|^2 = H^{2/3} \frac{4 \zeta(4/3)}{3 \zeta(2)} \int_{0}^{\infty} |w(y)|^2 y^{1/3} dy + O_\epsilon (H^{2/3 - \epsilon / 4} ).
\end{equation}

Proof of claim: This is very similar to Lemma 8 in \cite{GMRR}. Firstly, we complete the sum over all $b$. By \eqref{deriv} for $k = 0$ and $l = 0, 1$, we have
\[
\sum_{n \ge 1} |w(n / \nu)|^2 \ll \sum_{1 \le n \le \nu H^{\epsilon / 4}} 1 + \sum_{n > \nu H^{\epsilon / 4}} \frac{H^{\epsilon / 2} \nu^2}{n^2} \ll_\epsilon \nu H^{\epsilon / 4}
\]
for any $\nu > 0$. Hence,
\[
2 H^2 \sum_{b > H^{2/3 + \epsilon}} \frac{\mu^2(b)}{b^3} \sum_{n \ge 1} \Big| w\Bigl( \frac{n H}{b^{3/2}} \Bigr) \Big|^2 \ll_\epsilon H^{1 + \epsilon / 4} \sum_{b > H^{2/3 + \epsilon}} \frac{1}{b^{3/2}} \ll H^{2/3  - \epsilon/4}.
\]
Thus, the left hand side of \eqref{smoothW} is
\begin{equation} \label{smoothW2}
2 H^2 \sum_{b \ge 1} \frac{\mu^2(b)}{b^3} \sum_{n \ge 1} \Big| w\Bigl( \frac{n H}{b^{3/2}} \Bigr) \Big|^2 + O_\epsilon (H^{2/3  - \epsilon/4}).
\end{equation}

Secondly, we use contour integration to simplify \eqref{smoothW2}. Define $g(x) := |w(e^x)|^2 e^x$, which is smooth and decays exponentially as $|x| \rightarrow \infty$ by \eqref{deriv}. Its Fourier transform is given by
\[
\hat{g}(\xi) = \int_{-\infty}^{\infty} |w(e^x)|^2 e^x e^{-2 \pi i x \xi} dx = \int_{0}^{\infty} |w(y)|^2 y^{-2 \pi i \xi} dy,
\]
and one can show that (i) $\hat{g}(\xi)$ is entire and (ii) $\hat{g}(\xi) = O(H^{2\epsilon} / (|\xi| + 1)^3)$ uniformly for $|\Im(\xi)| < 1 / (2 \pi)$ by \eqref{deriv} and standard partial summation argument. Then Fourier inversion gives, for $r > 0$,
\[
|w(r)|^2 = r^{-1} \frac{1}{2 \pi i} \int_{(c)} r^s \hat{g} \Bigl( \frac{s}{2 \pi i} \Bigr) ds,
\]
where the integral is over $\Re(s) = c$ with $-1 < c < 1$. Hence, taking $c = - 1/4$,
\[
2 H^2 \sum_{b \ge 1} \frac{\mu^2(b)}{b^3} \sum_{n \ge 1} \Big| w\Bigl( \frac{n H}{b^{3/2}} \Bigr) \Big|^2 = \frac{H}{i \pi} \sum_{b \ge 1} \frac{\mu^2(b)}{b^{3/2}} \sum_{n \ge 1} \frac{1}{n} \int_{(-1/4)} H^s n^s b^{-3s/2} \hat{g}\Bigl( \frac{s}{2 \pi i} \Bigr) ds.
\]
Due to absolute convergence of the series, one can take the sums inside the integral and the above becomes
\[
\frac{H}{i \pi} \int_{(-1/4)}  H^s  \zeta(1 - s) \frac{\zeta(\frac{3}{2}(1 + s))}{\zeta(3(1+s))}  \hat{g}\Bigl( \frac{s}{2 \pi i} \Bigr) ds
\]
by \eqref{zeta}. Because of the bounds on $\hat{g}$ and $\zeta(1/2 + it) \ll t^{1/6} \log^2 t$, we can push the contour integral above to the left to an integral over $\Re(s) = -2/3$. Picking out a residue at $s = -1/3$ from the singularity of $\zeta(\frac{3}{2}(1 + s))$, the above simplifies to
\[
\frac{4 \zeta(4/3)}{3 \zeta(2)} \hat{g} \Bigl( - \frac{1}{6 \pi i} \Bigr) H^{2/3} + O(H^{1/3 + 2\epsilon}) =  \frac{4 \zeta(4/3)}{3 \zeta(2)} \int_{0}^{\infty} |w(y)|^2 y^{1/3} dy \cdot H^{2/3} + O(H^{1/3 + 2 \epsilon}).
\]
This and \eqref{smoothW2} finish the proof of the claim. To deal with \eqref{diagonal-main}, we introduce the function
\[
w(y) := S(y) h\Bigl( \frac{y}{H^{\epsilon/4}} \Bigr)
\]
where $h$ is a smooth bump function such that $h(x) = 1$ for $|x| \le 1$ and $h(x) = 0$ for $|x| \ge 2$ (for example, one can take $h(x) = \sigma_{10, 2, 0.1}^{+}(x + 3)$). Then one can check that $w(y)$ satisfies the hypothesis \eqref{deriv}. For such $w$,
\begin{align} \label{sum-error}
2H^2 \sum_{b \le H^{2/3 + \epsilon}} & \frac{\mu^2(b)}{b^3} \sum_{n \ge 1} \biggl( S\Bigl( \frac{n H}{b^{3/2}} \Bigr)^2 - w \Bigl( \frac{n H}{b^{3/2}} \Bigr)^2 \biggr) \ll H^2 \mathop{\sum_{b \le H^{2/3 + \epsilon}, n \ge 1}}_{n H^{1 - \epsilon/4} \ge b^{3/2}} \frac{1}{b^3} \frac{1}{(n H / b^{3/2})^2} \nonumber \\
&\ll \sum_{b \le H^{2/3 - \epsilon/6}} \sum_{n \ge 1} \frac{1}{n^2} + \sum_{b > H^{2/3 - \epsilon/6}} \sum_{n \ge b^{3/2} / H^{1 - \epsilon/4}} \frac{1}{n^2} \ll_\epsilon H^{2/3 - \epsilon/6}.
\end{align}
On the other hand,
\begin{equation} \label{integral-error}
\int_{0}^{\infty} |S(y)|^2 y^{1/3} dy - \int_{0}^{\infty} |w(y)|^2 y^{1/3} dy \ll \int_{H^{\epsilon/4}}^{\infty} y^{-5/3} dy \ll_\epsilon H^{-\epsilon / 6}.
\end{equation}
Combining \eqref{smoothW}, \eqref{sum-error} and \eqref{integral-error}, we have
\[
2 H^2 \sum_{b \le H^{2/3 + \epsilon}} \frac{\mu^2(b)}{b^3} \sum_{n \ge 1} \Big| S\Bigl( \frac{n H}{b^{3/2}} \Bigr) \Big|^2 = H^{2/3} \frac{4 \zeta(4/3)}{3 \zeta(2)} \int_{0}^{\infty} |S(y)|^2 y^{1/3} dy + O_\epsilon (H^{2/3 - \epsilon/6})
\]
which gives Proposition \ref{prop3}.

\bigskip

{\bf Acknowledgment:} The author would like to thank Zeev Rudnick and Yuk-Kam Lau for encouraging and helpful discussions during the HKU Number Theory Days conference on June 26-30, 2023 while part of this work was presented. He also wants to thank the Department of Mathematics of the University of Hong Kong for their hospitality.

\bibliographystyle{amsplain}

Mathematics Department \\
Kennesaw State University \\
Marietta, GA 30060 \\
tchan4@kennesaw.edu

\end{document}